\documentclass[letterpaper,12pt,oneside]{article}
\usepackage{amsmath,amsthm}
\usepackage{amsfonts}
\usepackage{latexsym}
\usepackage{xcolor}
\usepackage{multicol}
\usepackage{enumitem}
\usepackage[latin1]{inputenc}
\usepackage{amssymb}
\usepackage{oldgerm}
\usepackage{array}
\usepackage{hyperref}
\theoremstyle{plain}

\newtheorem{Conjec}{Conjecture}

\newtheorem{lema}{Lemma}
\newtheorem{teor}{Theorem}

\def\A{\alpha}

\def\C{\mathbb{C}}

\title{Identities for the $ k$--generalized Fibonacci sequence with negative indices and its zero--multiplicity}

\author{Jonathan Garc\'ia \footnote{Departamento de Matem\'aticas, Universidad del Valle, 25360 Cali, Calle 13 No 100-00, Colombia. 
e-mail: garcia.jonathan@correounivalle.edu.co}\,\,\,\,
Carlos A. G\'omez \footnote{Departamento de Matem\'aticas, Universidad del Valle, 25360 Cali, Calle 13 No 100-00, Colombia; Research Group ALTENUA--Universidad del Valle: Álgebra, Teoría de Números y Aplicaciones, Colciencias code: COL0017217.
e-mail: carlos.a.gomez@correounivalle.edu.co} \,\,\,\, Florian Luca \footnote{School of Mathematics, University of the Witwatersrand, Johannesburg, South Africa; Research Group in Algebraic Structures and Applications, King Abdulaziz University, Jeddah, Saudi Arabia; Max Planck Institute for Software Systems Saarbr\"ucken, Germany and Centro de Ciencias Matem\'aticas UNAM, Morelia, Mexico.  
e-mail: florian.luca@wits.ac.za}}

\begin{document}

\maketitle

\begin{abstract}
In this paper, we prove identities for members of the $k$--generalized Fibonacci sequence with 
negative indices and we apply these identities to deduce an exact formula for 
its zero--multiplicity.

\end{abstract}

\emph{Key words and phrases:} $k$--generalized Fibonacci sequence, zero--multiplicity of linear recurrences.

\emph{Mathematics Subject Classification 2020}: 11B39, 11J86

\section{Introduction}

One of the most famous and curious numerical sequence, due to the large number of properties and relationships with other areas \cite{Koshy}, is the Fibonacci sequence, denoted by ${\bf F}:=(F_n)_{n\geq 0}$. Its initial values are $F_0=0$, $F_1=1$ and obeys the recurrence $F_{n}=F_{n-1}+F_{n-2}$ for all $n\geq 2$. This sequence has been generalized in many ways, some by preserving the initial conditions, and others by preserving the recurrence relation. In general, a sequence $(u_n)_{n\in {\mathbb Z}} \subseteq \mathbb{C} $ is a linear recurrence sequence of order  $ k\in\mathbb{Z}^{+} $ if it satisfies the recurrence relation $u_{n+k}=a_{1}u_{n+k-1}+a_{2}u_{n+k-2}+\cdots+a_{k}u_{n}$ for all $n\geq 0$ with coefficients $ a_1,\ldots,a_k\in \mathbb{C} $ and $a_k\ne 0$. We assume that $k$ is minimal with the above property. Such a sequence $(u_{n})_{n\in {\mathbb Z}} $ has an associated characteristic polynomial given by
$\Psi_{k}(z) = z^{k} -a_{1}z^{k-1}-\cdots-a_{k}$. Let $\alpha_1, \ldots, \alpha_s \in \C$ be the distinct roots of $\Psi_{k}(z)$. From the theory of linear recurrence sequences (see~\cite[Theorem C.1]{T.N.Shorey2008}) there are polynomials $p_1(z), p_2(z), \ldots, p_s(z)$ in ${\mathbb C}[z]$, uniquely determined by the initial values $u_0, \ldots, u_{k-1}$, such that the formula
$u_n = p_1(n)\alpha_1^{n} + p_2(n)\alpha_2^{n} + \cdots + p_s(n)\alpha_s^n$  holds for all  $n\in {\mathbb Z}$. Studying certain arithmetic properties of members of linear recurrence sequences $(u_n)$ can be performed 
with methods from the theory of exponential Diophantine equations by using the above expression for $u_n$.

In this paper, we present some identities for the members of the {\it $k$--generalized Fibonacci sequence} with negative indices. The $k$--generalized Fibonacci sequence $F^{(k)}$ is the linear recurrence sequence of order $k$ with initial values $F_{-i}^{(k)} = 0$ for $i = 0, 1, \ldots, k-2$ and $F_{1}^{(k)} = 1$, where each subsequent term is the sum of the $k$ previous ones $F_{n}^{(k)} =F_{n-1}^{(k)}+F_{n-2}^{(k)}+\cdots+F_{n-k}^{(k)}$, for all $n,k\ge 2$. This sequence  can be extended 
to all integer indices $n$. Since $F_n^{(k)}>0$ for $n>0$, all zeros of ${\bf F}^{(k)}$ correspond to non--positive indices $n$. We denote $H^{(k)}:=(H_{n}^{(k)})_{n\ge0}$, where $H_n^{(k)}=F_{-n}^{(k)}$. This sequence obeys the recurrence relation
\begin{equation}\label{rr}
H_{n}^{(k)} =H_{n-k}^{(k)}-H_{n-(k-1)}^{(k)}-\cdots-H_{n-1}^{(k)}\quad {\text{\rm for~all}}\quad n\ge k,
\end{equation}
with initial values $ H_{i}^{(k)}=0 $ for $ 0\leq i\leq k-2 $ and $ H_{k-1}^{(k)}=1 $. 

In this paper, we find identities for members of $H^{(k)}$ and apply them to deduce an exact formula for the zero--multiplicity of $H^{(k)}$.   

\section{The main results}

\subsection{Identities for $H^{(k)}$}
Our identities are similars to the given by Ferguson \cite{Ferguson66}, Gabai~\cite{Gabai} or Cooper and Howard~\cite{Cooper-Howard} for the $k$--generalized Fibonacci sequence with positive indices.
\begin{teor}\label{MT}
	For all $k \ge 2$, the sequence $H^{(k)}$ satisfies 
	\begin{enumerate}
		\item[$ (i) $] For all \,$ 0\leq m\leq r \leq k-2 $, we have $ H_{mk+r}^{(k)}=0 $.
		\item[$ (ii) $] For all \,$ m\in[1,k-1]$, we have $ H_{mk-1}^{(k)}=2^{m-1}$.
		\item[$ (iii) $]  For all \,$ 0\leq r<m\leq k-1 $,
		\[
		H_{mk+r}^{(k)}=(-1)^{r+1}\left[\binom{m-1}{r}+\binom{m}{r+1}\right]2^{m-r-2}.
		\]
		\item[$ (iv) $] For all $ r\in[-1,k-2] $ and $ m\geq k-1 $,
		\[
		H_{mk+r}^{(k)}=\sum_{i=0}^{l} (-1)^{ik+r+1}\left[\binom{m-i-1}{ik+r}+\binom{m-i}{ik+r+1}\right]2^{m-i(k+1)-r-2},
		\]
		where $ l=m-1 $ if $ k=2 $, and $ l =  \left\lfloor m/(k-1)\right\rfloor $ if $ k>2 $.
	\end{enumerate}
\end{teor}

\subsection{Zero--multiplicity for $H^{(k)}$}

One of the classical problems in the theory of linear recurrence sequences is the {\it Skolem problem}: Given the linear recurrence sequence ${\bf u}:=(u_n)_{n\in {\mathbb Z}}$, one want to find
${\mathcal Z}({\bf u})=\{n\in {\mathbb Z}: u_n=0\}$. The cardinality of ${\mathcal Z}({\bf u})$ (when this is finite) is called the {\it zero--multiplicity} of ${\bf u}$. There is no known algorithm to find ${\mathcal Z}({\bf u})$ in general. One of the more important results here is due to Skolem~\cite{Skolem, T.N.Shorey2008}: If the coefficients $a_1,\ldots,a_k$ of the linear recurrence sequence ${\bf u}$ are rational, then the set ${\mathcal Z}({\bf u})$ is a union of finitely many arithmetical progressions together with a finite set. Hagedorn~\cite{Hagedorn} showed that if the roots of $\Psi_k(z)$ 
are real then $\#{\mathcal Z}({\bf u})\le 2k-3$.

Our current research is motived by our previous work \cite{zero-multiplicity}. There we obtained that
\begin{equation}\label{zeromultcont}
\mathcal{A}:=\bigcup_{m=0}^{k-2} [m(k+1),(m+1)k-2] \subseteq \mathcal{Z}({H^{(k)}}).
\end{equation}
Thus, $\#{\mathcal Z}({ H^{(k)}}) \ge k(k-1)/2$. Furthermore, we checked that 
$$
\#{\mathcal Z}({H^{(2)}}) = 1,~ \#{\mathcal Z}({H^{(3)}}) = 4 \quad {\rm and} \quad \#{\mathcal Z}({H^{(k)}}) = k(k-1)/2,
$$ 
for $k \in [4, 500]$.

Based on the above results, we proposed in \cite{zero-multiplicity} the following conjecture.
\begin{Conjec}\label{conjectute}
The zero--multiplicity $\#{\mathcal Z}({H^{(k)}})$ of the $k$--generalized Fibonacci sequence with non-positive indices $H^{(k)}$ for $ k\geq4 $ is the $(k-1)$st triangular number; i.e. 
\begin{equation*}
\#{\mathcal Z}({H^{(k)}})= k(k-1)/2.
\end{equation*}
\end{Conjec}

Recently, relating the $2$--adic valuation of $F^{(k)}$ with the Diophantine equation $H_{n}^{(k)} = 0$, Young \cite{Young22} showed that for all $k>500$, 
$$
\#{\mathcal Z}({H^{(k)}}) \le k(k+1)/2 + \lfloor k/2\rfloor. 
$$ 
In this paper, we confirm the above conjecture. As a consequence of this, we get that ${\mathcal Z}({\bf F^{(k)}})$ is exactly $\mathcal{A}$.

\section{Preliminary results}
To simplify the notation, from now on we denote $ H_{n}^{(k)}:= H_{n} $, where $ k $ is fixed.
From the recurrence relation \eqref{rr}, it is easy to see that  
\begin{align}
	H_{n} &=2H_{n-k}-\left(H_{n-k}+H_{n-(k-1)}+\cdots+H_{n-2}+H_{n-1}\right)\nonumber\\
	&=2H_{n-k}-H_{n-k-1}\quad\text{for all}\quad  n\geq k+1.\label{identiti2H}
\end{align}
The containment~\eqref{zeromultcont} is a direct consequence  of the initial values of $ H_{n} $ and the identity~\eqref{identiti2H}. Indeed, inductively we see that  $ H_{n}=0 $ for all $ 0\leq n\leq k-2 $. Further
\begin{align*}
H_n&=2H_{n-k}-H_{n-k-1}=0 \quad\text{for all}\quad  m(k+1)\leq n\leq mk+(k-2),
\end{align*}
with $1 \le m \le k-2$. Moreover, the length of these intervals is given by the decreasing quantity 
$$
mk+(k-2)-m(k+1)=k-(m+2),
$$
which becomes $0$ at $ m=k-2 $. 

So far, we have only summarized what we obtained in \cite{zero-multiplicity}. This is fundamental for what follows, since we note that the sequence $ H_{n} $ can be represented as lists matrices, starting with a rectangular matrix of size $ (k-1)\times(k+1) $ and continuing with square matrices of size $ k\times k $.
Then, we observed interesting patterns in $ H_{n} $ that we were able to formulate and prove.\\

We start by ordering the first $ k^2-1 $ elements of the sequence $ H_{n} $ in matrix form $\left(H_{(i-1)(k+1)+(j-1)}\right)_{ij} $  with $ 1\leq i\leq k-1 $ and $ 1\leq j\leq k+1  $. This is
\[
	\left(\begin{array}{llllll}
		\bf{H_{0}} & \cdots & \bf{H_{k-2}} & H_{k-1} & H_{k}\\
		\bf{H_{(k+1)+0}} & \cdots & H_{(k+1)+k-2} & H_{(k+1)+k-1} & H_{(k+1)+k}\\
		\bf{H_{2(k+1)+0}} & \cdots & H_{2(k+1)+k-2} & H_{2(k+1)+k-1} & H_{2(k+1)+k}\\
		\quad\vdots & \ddots & \quad\vdots & \quad\vdots & \quad\vdots\\
		\bf{H_{(k-2)(k+1)+0}} & \cdots & H_{(k-2)(k+1)+k-2} & H_{(k-2)(k+1)+k-1} & H_{(k-2)(k+1)+k}
	\end{array}\right)
\]
and we notice that the ``upper triangular'' part, which is in bold, is composed of zeros.
We can also write this matrix as 
\begin{small}
	\begin{equation}\label{matrix2}
		\left(\begin{array}{lllllll}
			0 & 0 & \cdots & 0 & 0 & H_{k-1} & H_{k}\\
			0 & 0 & \cdots & 0 & H_{2k-1} & H_{2k} & H_{2k+1}\\
			0 & 0 & \cdots & H_{3k-1} & H_{3k} & H_{3k+1} & H_{3k+2}\\
			\vdots & \vdots & \ddots & \quad\vdots & \quad\vdots & \quad\vdots & \quad\vdots\\
			0 & H_{(k-1)k-1} & \cdots & H_{(k-1)k+k-5} & H_{(k-1)k+k-4} & H_{(k-1)k+k-3} & H_{(k-1)k+k-2}
		\end{array}\right),
	\end{equation}
\end{small}
and we observe the following behavior in the non-zero diagonals. 
\begin{lema}\label{lema0}
	~
	\begin{enumerate}
		\item[$ (a) $] For all $ m\in[1,k-1] $, we have $ H_{mk-1}=2^{m-1}$.
		\item[$ (b) $] For all $ m\in[1,k-1] $, we have $ H_{mk}=-(m+1)2^{m-2}$.
		\item[$ (c) $] For all \, $  1\leq r<m\leq k-1  $,  \[H_{mk+r}=-\sum_{j=r}^{m-1} 2^{m-1-j}H_{jk+r-1}.\]
	\end{enumerate}
\end{lema}
\begin{proof} 
	To show that the items of this lemma are fulfilled, we proceed by induction on $ m $, taking into account that the upper bound for $m$ is $ k-1\geq1 $. 
	\begin{enumerate}[leftmargin=*]
		\item[$ (a) $] If $ m=1 $, given the initial values, we have $ H_{k-1}=1=2^{0} $. Suppose that $ H_{(m-1)k-1}=2^{m-2} $ for $ m\in[2,k-1] $. Then, using identity~\eqref{identiti2H}, it follows that 
		\[
			H_{mk-1}=2H_{(m-1)k-1}-H_{(m-1)k-2}=2^{m-1},
		\]
		since $ 0\leq m-2\leq k-3 $ and therefore $ H_{((m-2)+1)k-2}=0 $ (see \eqref{zeromultcont}).
		
		\item[$ (b) $] When $ m=1 $, we have
		\[
			H_{k}=H_{0}-H_{1}-\cdots-H_{k-1}=-1=-2\times2^{-1}.
		\]
		Assume  $ H_{(m-1)k}=-m2^{m-3} $ and we obtain from identity~\eqref{identiti2H} and item $ (a) $, that 
		\[
			H_{mk}=2H_{(m-1)k}-H_{(m-1)k-1}=-m2^{m-2}-2^{m-2}=-(m+1)2^{m-2}.
		\] 
		\item[$ (c) $] Assume $ m=r+1 $. We have by identity~\eqref{identiti2H} that
		\[
			H_{(r+1)k+r}=2H_{rk+r}-H_{rk+r-1}=-H_{rk+r-1}=-\sum_{j=r}^{r} 2^{r-j}H_{jk+r-1},
		\]
		because $ H_{rk+r}=H_{r(k+1)}=0 $ with $ 1\leq r\leq k-2 $ (see \eqref{zeromultcont}). We take 	
		\[
			H_{(m-1)k+r}=-\sum_{j=r}^{m-2} 2^{m-2-j}H_{jk+r-1}
		\]
		as the inductive hypothesis. Then 
		\begin{align*}
				H_{mk+r}&=2H_{(m-1)k+r}-H_{(m-1)k+r-1}\\&=-\left(\sum_{j=r}^{m-2} 2^{m-1-j}H_{jk+r-1}\right)-H_{(m-1)k+r-1}\\
				&=-\sum_{j=r}^{m-1} 2^{m-1-j}H_{jk+r-1}.\qedhere
		\end{align*}
	\end{enumerate}
\end{proof}
Matrix~\eqref{matrix2} is the first element of the list of matrices that we use to organize all the elements of $ H_n $, as we mentioned before. It is rectangular of size $ (k-1)\times(k+1) $. The matrix that follows is of size $ k \times k $ and includes all the non-zero elements of the last row of matrix~\eqref{matrix2}, this is 
\begin{equation}\label{matrix3}
	\left(\begin{array}{llll}
		H_{(k^{2}-k-1)+0} & H_{(k^{2}-k-1)+1} & \cdots & H_{(k^{2}-k-1)+k-1}\\
		H_{(k^{2}-k-1)+k} & H_{(k^{2}-k-1)+k+1} & \cdots & H_{(k^{2}-k-1)+k+(k-1)}\\
		\quad\vdots & \quad\vdots & \ddots & \quad\vdots\\
		H_{(k^{2}-k-1)+(k-1)k} & H_{(k^{2}-k-1)+(k-1)k+1} & \cdots & H_{(k^{2}-k-1)+(k-1)k+(k-1)}
	\end{array}\right).
\end{equation}
In general, all square matrices after matrix~\eqref{matrix2} form the sequence $ \{M_b\}_{b\in\mathbb{Z}^{+}} $, where
\[
	 M_b :=	\left(\begin{array}{llll}
		H_{(bk^{2}-bk-1)+0} & H_{(bk^{2}-bk-1)+1} & \cdots & H_{(bk^{2}-bk-1)+k-1}\\
		H_{(bk^{2}-bk-1)+k} & H_{(bk^{2}-bk-1)+k+1} & \cdots & H_{(bk^{2}-bk-1)+k+(k-1)}\\
		\quad\vdots & \quad\vdots & \ddots & \quad\vdots\\
		H_{(bk^{2}-bk-1)+(k-1)k} & H_{(bk^{2}-bk-1)+(k-1)k+1} & \cdots & H_{(bk^{2}-bk-1)+(k-1)k+(k-1)}
	\end{array}\right).
\]
 Furthermore, by calling $ M_{0} $ the matrix~\eqref{matrix2}, we have arrived at
\[
\left(H_{n}^{(k)}\right)_{n\geq0}=\bigcup_{b\geq0}\left\{h\,:\,h \text{ is an entry of } M_{b}\right\}.
\]
From now on, we simplify the notation by setting
\[
H_{b,jk+r}:=H_{(bk^2-bk-1)+jk+r}\quad\text{for all }~ b\geq0.
\]
Note that matrix~\eqref{matrix3} is exactly $ M_1 $. In fact, these matrices satisfy that their first row is exactly the last row of the immediately preceding matrix. Indeed,
\begin{equation}
	H_{b,r}=H_{b-1,(k-1)k+r}\quad\text{for all } ~0\leq r\leq k-1 \quad\text{and }~ b\geq1.  \label{roweqrow1}
\end{equation}
Also, 
\begin{equation}
	H_{b,-1}=H_{b-1,(k-2)k+(k-1)} \quad\text{ for all }~ b\geq1\label{consecnotationH0}
\end{equation}
and in particular, by containment~\eqref{zeromultcont}, 
\begin{equation*}
	 H_{1,-1}=H_{0,(k-2)k+(k-1)}=H_{(k-2)(k+1)}=0.
\end{equation*}
Moreover, we observe that identity~\eqref{identiti2H} is still satisfied with this notation 
\begin{equation}
	H_{b,n}=2H_{b,n-k}-H_{b,n-k-1} \label{identiti2Hb}
\end{equation}
for all $ b\geq1 $ and $ n\geq2 $, or for $ b=0 $ and $ n\geq k+2 $.\\

 Next, we find the following patterns in all entries below the main diagonal of $ M_{b} $ with $ b\geq1 $. 
\begin{lema}\label{lemaproperty}
	Let $ b\geq1 $, then
	\begin{enumerate}[leftmargin=*]
		\item[$(I)$] $ H_{b,jk}=\begin{cases}
			2H_{b,0}-H_{b-1,(k-2)k+(k-1)}, & \text{ if }~j=1,\\
			2^{j-1}H_{b,k}-\sum_{i=0}^{j-2}2^{j-i-2}H_{b,(i+1)k-1}, & \text{ if }~2\leq j\leq k-1.
		\end{cases} $
		
		\item[$(II)$] For all $r, j\in[1,k-1]$, it holds that
		\[
			H_{b,jk+r}=2^{j-1}H_{b,k+r}-\sum_{i=1}^{j-1}2^{j-1-i}H_{b,ik+r-1}.
		\]
		\item[$(III)$] Let $ r,j\in[0,k-1] $ be fixed. Then
		\[
		H_{b,jk+r}=\sum_{i=0}^{t}(-1)^i \binom{t}{i}2^{t-i} H_{b,(j-t)k+r-i},
		\]
		for all $  t\in \left[0,\left\lfloor \frac{bk^2-bk-1+jk+r}{k+1} \right\rfloor \right]  $.
	\end{enumerate}	
\end{lema}

\begin{proof} We consider each item separately. 
	\begin{enumerate}[leftmargin=*]
		\item[$ (I) $] The case $ j=1 $, follows from identities \eqref{consecnotationH0} and \eqref{identiti2Hb}. Let $ 2\leq j\leq k-1 $. We argue by induction on $ j $. If $ j=2 $, item $ (I) $ is immediate by identity~\eqref{identiti2Hb}. Suppose it is satisfied for $ j-1 $; i.e.,
		\begin{equation}
			H_{b,(j-1)k}=2^{j-2}H_{b,k}-\sum_{i=0}^{j-3}2^{j-i-3}H_{b,(i+1)k-1}.\label{hpinlemaHb1}
		\end{equation}
		Then, by \eqref{identiti2Hb} and \eqref{hpinlemaHb1}, 
		\begin{align*}
			H_{b,jk}&=2H_{b,(j-1)k}-H_{b,(j-1)k-1}\\
			&=\left(2^{j-1}H_{b,k}-\sum_{i=0}^{j-3}2^{j-i-2}H_{b,(i+1)k-1}\right)-H_{b,(j-1)k-1}\\
			&=2^{j-1}H_{b,k}-\sum_{i=0}^{j-2}2^{j-i-2}H_{b,(i+1)k-1},
		\end{align*}
		and thus we conclude item $ (I) $ for $j$.
		
		\item[$ (II) $] If $ j=1 $, the identity is trivial. For $ j\geq2 $, we apply recursively, $ j-1 $ times, identity~\eqref{identiti2Hb} on $ H_{b,jk+r} $. That is,
		\begin{align*}
			H_{b,jk+r}&=2H_{b,(j-1)k+r}-H_{b,(j-1)k+r-1}\\
			&=2^2H_{b,(j-2)k+r}-2H_{b,(j-2)k+r-1}-H_{b,(j-1)k+r-1}\\
			&=2^3H_{b,(j-3)k+r}-2^2H_{b,(j-3)k+r-1}-2H_{b,(j-2)k+r-1}-H_{b,(j-1)k+r-1}\\
			&\;\;\vdots\\
			&=2^{j-1}H_{b,(j-(j-1))k+r}-2^{j-2}H_{b,(j-(j-1))k+r-1}-\cdots-H_{b,(j-1)k+r-1}.
		\end{align*}
	
		\item[$ (III) $] We proceed by induction on $ t $. If $ t=0 $, the identity is trivial. Suppose it is satisfied for $ t-1\geq 0 $; i.e., 
		\[
		H_{b,jk+r}=\sum_{i=0}^{t-1}(-1)^i \binom{t-1}{i}2^{t-i-1} H_{b,(j-t+1)k+r-i}.
		\]
		Then, by the inductive hypothesis and identity~\eqref{identiti2Hb}, it follows that 
\begin{align*}
			&H_{b,jk+r}=\sum_{i=0}^{t-1}(-1)^i \binom{t-1}{i}2^{t-i-1} \left(2H_{b,(j-t)k+r-i}-H_{b,(j-t)k+r-i-1}\right)\\
			&=(-1)^0 \binom{t-1}{0}2^{t}H_{b,(j-t)k+r}+ \sum_{i=1}^{t-1}(-1)^i \binom{t-1}{i}2^{t-i}H_{b,(j-t)k+r-i}\\
			&\qquad\qquad\qquad\qquad\qquad\qquad~~+\sum_{i=1}^{t}(-1)^{i} \binom{t-1}{i-1}2^{t-i}H_{b,(j-t)k+r-i}
\end{align*}
\begin{align*}			
			&= 2^{t}H_{b,(j-t)k+r}+ \sum_{i=1}^{t-1}(-1)^i \left(\binom{t-1}{i}+\binom{t-1}{i-1}\right)2^{t-i}H_{b,(j-t)k+r-i}\\
			&\qquad\qquad\qquad~~+(-1)^{t} \binom{t-1}{t-1}2^{t-t}H_{b,(j-t)k+r-t}\\
			&= 2^{t}H_{b,(j-t)k+r}+ \sum_{i=1}^{t-1}(-1)^i \binom{t}{i}2^{t-i}H_{b,(j-t)k+r-i}+(-1)^{t}H_{b,(j-t)k+r-t}\\
			&=\sum_{i=0}^{t}(-1)^i \binom{t}{i}2^{t-i}H_{b,(j-t)k+r-i},
\end{align*}
		and the identity is satisfied for $ t $. We must assume that $ t\leq \left\lfloor \frac{bk^2-bk-1+jk+r}{k+1} \right\rfloor $, because in this case
		\[
			 bk^2-bk-1+(j-t)k+r-i\geq bk^2-bk-1+(j-t)k+r-t\geq 0,
		\]
for all $0\le i\le t $. Therefore, $ H_{b,(j-t)k+r-i} $ is well defined for all $0\le i\le t $.
		
This completes the proof of this lemma. \qedhere
\end{enumerate}
\end{proof}

The above lemmas will allow us, in Section~\ref{proofMT}, to characterize all entries of the matrices $ M_b $ and therefore all elements of $ H_{n} $. For this purpose we introduce the following notation:
\begin{equation}\label{notationpsi}
	\psi(v,w):=\binom{v}{w}+\binom{v+1}{w+1}.
\end{equation}
Note that 
\begin{equation}
	\psi(v,-1)=1 ~\text{ for all }~v\neq -1~~\text{ and }~~\psi(v,w)=0 ~\text{ for all }~w>v\geq0.\label{eq12not}
\end{equation}

We prove the following properties of $\psi(v,w)$.
\begin{lema}\label{lemaproperpsi}
	The funtion $\psi$ satisfies:
	\begin{enumerate}
		\item [$(1)$] $ \psi(v,w)+\psi(v,w+1)=\psi(v+1,w+1) $.	
		\item [$(2)$] $ \sum_{i=0}^{n} \psi(v+i,v)=\psi(v+n+1,v+1)$.
		\item [$(3)$]	$ \sum_{i=-1}^{n} \psi(v+i,i)=\psi(v+n+1,n)$.
		\item [$(4)$]	$ \sum_{i=w}^{v} \psi(i-1,w-1)=\psi(v,w)$.
		\item [$(5)$] $ \sum_{i=0}^{w}\binom{u}{i}\psi(v,w-i-1)=\psi(u+v,w-1)$.
	\end{enumerate}
\end{lema}
\begin{proof} To show that all items are satisfied, we use basic properties of binomial coefficients. Item $(1)$ is immediate. 
\begin{enumerate}[leftmargin=*]
\item [$(2)$] We know that 
		\begin{align*}
			\sum_{i=0}^{n} \psi(v+i,v)&=\sum_{i=v}^{v+n} \binom{i}{v} +\sum_{i=v+1}^{v+n+1} \binom{i}{v+1}=\psi(v+n+1,v+1).
		\end{align*}
		\item[$ (3) $] It is clear that
		\begin{align*}
			\sum_{i=-1}^{n} \psi(v+i,i)&=\sum_{i=0}^{n} \binom{v+i}{i} +\sum_{i=-1}^{n} \binom{v+i+1}{i+1} \\
			&=\sum_{i=v}^{v+n} \binom{i}{v} +\sum_{i=v}^{v+n+1} \binom{i}{v}\\
			&=\psi(v+n+1,n).
		\end{align*}
		\item[$ (4) $] We have
		\begin{align*}
			\sum_{i=w}^{v} \psi(i-1,w-1)&=\sum_{i=w-1}^{v-1} \binom{i}{w-1} +\sum_{i=w}^{v} \binom{i}{w}=\psi(v,w).
		\end{align*}
		\item[$ (5) $] We obtain
		\begin{align*}
			\sum_{i=0}^{w}\binom{u}{i}\psi(v,w-i-1)&=\sum_{i=0}^{w-1}\binom{u}{i}\binom{v}{w-i-1}+\sum_{i=0}^{w}\binom{u}{i}\binom{v+1}{w-i}\\
			&=\psi(u+v,w-1).\qedhere
		\end{align*}
	\end{enumerate}
\end{proof}

\section{Proof of Theorem \ref{MT}}\label{proofMT}

For items $ (i) $ and $ (ii) $, see containment~\eqref{zeromultcont} and Lemma~\ref{lema0}, item ($a$), respectively. The remaining items are shown below.

\subsection{Item $ (iii) $} 
With notation~\eqref{notationpsi}, we must prove that 
\begin{equation}\label{itemiiiequiv}
	H_{mk+r}=(-1)^{r+1}\psi(m-1,r)2^{m-r-2}\qquad\text{for all}~~~0\leq r<m\leq k-1.
\end{equation}
For this, we apply induction on $ r $. If $ r=0 $, by Lemma~\ref{lema0}, item ($b$), we get that identity~\eqref{itemiiiequiv}  is satisfied. Indeed,
\[
H_{mk}=-(m+1)2^{m-2}=(-1)^{1}\psi(m-1,0)2^{m-2}\qquad\text{for all}~~~1\leq m\leq k-1.
\]
We assume by the  inductive hypothesis that identity~\eqref{itemiiiequiv}  is satisfied for $ r-1 $ and therefore $ 0\leq r-1< m\leq k-1 $; i.e.,
\[
H_{mk+r-1}=(-1)^{r}\psi(m-1,r-1)2^{m-r-1}\qquad\text{for all}~~~r\le m \le k-1.
\]
In particular,
\begin{equation}
		H_{jk+r-1}=(-1)^{r}\psi(j-1,r-1)2^{j-r-1}\qquad\text{for all}~~~r \le j \le m-1.\label{eqhpinduciii}
\end{equation}
Then, by Lemma~\ref{lema0}, item ($c$), Lemma~\ref{lemaproperpsi}, item ($4$) and identity~\eqref{eqhpinduciii}, we obtain  
\begin{align*}
	H_{mk+r}&=-\sum_{j=r}^{m-1} 2^{m-1-j}\left((-1)^{r}\psi(j-1,r-1)2^{j-r-1}\right)\\
	&=(-1)^{r+1}\left(\sum_{j=r}^{m-1}\psi(j-1,r-1)\right) 2^{m-r-2}\\
	&=(-1)^{r+1}\psi(m-1,r)2^{m-r-2}\qquad\text{for all}~~~1\leq r<m\leq k-1.
\end{align*}
This completes the proof of item $ (iii) $.

\subsection{Item $ (iv) $} 

\subsubsection{Case $ k=2 $:} Here we must prove that for all $ m\geq 1 $,
\begin{equation}\label{eqcaseivs21}
		H_{2m+r}=\begin{cases}
		\sum_{i=0}^{m-1}\psi(m-i-1,2i-1)2^{m-3i-1}, & \text{ if }~r=-1,\\
		-\sum_{i=0}^{m-1}\psi(m-i-1,2i)2^{m-3i-2},  & \text{ if }~r=0.
	\end{cases}
\end{equation}
For this we proceed by induction on $ m $. Suppose that identity~\eqref{eqcaseivs21} is satisfied up to $ m-1\geq2 $ (the cases $ m=1 $ and $ m=2 $ are easily verified). Then using~\eqref{identiti2H}, \eqref{eq12not}, Lemma~\ref{lemaproperpsi} and the fact that $ \psi(0,2m-2)=\psi(0,2m-3)=0 $ (since $ m\geq 2 $), we obtain
\begin{small}
	\begin{align*}
		&H_{2m-1}=2H_{2(m-1)-1}-H_{2(m-2)}\\
		&=\sum_{i=0}^{m-2}\psi(m-i-2,2i-1)2^{m-3i-1}+\sum_{i=0}^{m-3}\psi(m-i-3,2i)2^{m-3i-4}\\
		&=\psi(m-2,-1)2^{m-1}+\sum_{i=1}^{m-2}(\psi(m-i-2,2i-1)+\psi(m-i-2,2i-2))2^{m-3i-1}\\
		&=\psi(m-1,-1)2^{m-1}+\sum_{i=1}^{m-2}\psi(m-i-1,2i-1)2^{m-3i-1}+\psi(0,2m-3)2^{-2(m-1)}\\
		&=\sum_{i=0}^{m-1}\psi(m-i-1,2i-1)2^{m-3i-1},
	\end{align*}
\end{small}
and 
\begin{small}
		\begin{align*}
		&H_{2m}=2H_{2(m-1)}-H_{2(m-1)-1}\\
		&=-\sum_{i=0}^{m-2}\psi(m-i-2,2i)2^{m-3i-2}-\sum_{i=0}^{m-2}\psi(m-i-2,2i-1)2^{m-3i-2}\\
		&=-\psi(0,2m-2)2^{-2m+1}-\sum_{i=0}^{m-2}\left(\psi(m-i-2,2i)+\psi(m-i-2,2i-1)\right)2^{m-3i-2}\\
		&=-\sum_{i=0}^{m-1}\psi(m-i-1,2i)2^{m-3i-2}.
	\end{align*}
\end{small}

\subsubsection{Case $ k>2 $:}

To treat this case, we present the following two lemmas.

\begin{lema}\label{lemaitemiv00}
	For all $ k>2 $ and $ j,r\in[0,k-1] $, it holds that
	\begin{align}
		H_{1,jk+r}&=(-1)^{r} 2^{k+j-r-2} \psi(k+j-2,r-1)\nonumber\\
		&+(-1)^{k+r} 2^{j-r-3} \psi(k+j-3,k+r-1).\label{firtstepinducb}
	\end{align}
\end{lema}
\begin{proof} We apply double induction on $ j,r\in[0,k-1] $.
	\begin{enumerate}[leftmargin=*]
		\item[$ (A) $] Let $ j=r=0 $. Since $ k>2 $, by observation~\eqref{eq12not}, it follows that
		\[
		\psi(k-2,-1)=1 \quad\text{ and }\quad \psi(k-3,k-1)=0.
		\]
		Then, using item $ (ii) $, we obtain
		\[
		H_{1,0}=H_{(k-1)k-1}=2^{k-2}=2^{k-2} \psi(k-2,-1)+(-1)^{k} 2^{-3} \psi(k-3,k-1)
		\] 
		fulfilling identity~\eqref{firtstepinducb}.
		
		
		\item[$ (B) $] Let $ j=0 $ and $ r\in[1,k-1] $. Then, by item $ (iii) $, we get that
		\begin{equation*}
			H_{1,r}=H_{k(k-1)+r-1}=(-1)^{r}\psi(k-2,r-1)2^{k-r-2}
		\end{equation*}
		and by observation~\eqref{eq12not} it follows that 
		\begin{align}
			H_{1,r}&=(-1)^{r} 2^{k-r-2}\psi(k-2,r-1)\label{eqitemB1r}\\
			&=(-1)^{r} 2^{k-r-2} \psi(k-2,r-1)+(-1)^{k+r} 2^{-r-3} \psi(k-3,k+r-1).\nonumber
		\end{align}
		Thus, identity~\eqref{firtstepinducb} is satisfied.
		
		
		\item[$ (C) $] Let $ r=0 $ and $ j\in[1,k-1] $. If $ j=1 $, by Lemma~\ref{lemaproperty}, item ($I$) and observation~\eqref{eq12not}, we have that identity~\eqref{firtstepinducb} is satisfied. In fact, 
		\begin{align}
			H_{1,k}&=2H_{1,0}-H_{0,(k-2)k+(k-1)}=2H_{1,0}=2^{k-1}\label{hiducitemC}\\
			&=2^{k-1} \psi(k-1,-1)+(-1)^{k} 2^{-2} \psi(k-2,k-1).\nonumber
		\end{align}	
		If $ j\in[2,k-1] $, we note that by Lemma~\ref{lemaproperty}, item ($I$), combined with item $(III)$ \footnote{This item is used with $ j=i $, $ r=k-1 $ and $ t=i $ for the elements $ H_{1,ik+k-1} $.} of this same Lemma and items $ (A) $, 
		$ (B) $ above, we get
		\begin{align}
			H_{1,jk}&=2^{j-1}H_{1,k}-\sum_{i=0}^{j-2}2^{j-i-2}H_{1,ik+k-1}\nonumber\\
			&=2^{j-1}H_{1,k}+\sum_{i=0}^{j-2}\left(\sum_{\ell=0}^{i}(-1)^{\ell+1} \binom{i}{\ell}2^{j-2-\ell} H_{1,k-1-\ell}\right)\nonumber\\
			&=2^{j-1}H_{1,k}+\left(\sum_{\ell=0}^{j-2}(-1)^{\ell+1} \left(\sum_{i=\ell}^{j-2}\binom{i}{\ell}\right)2^{j-2-\ell} H_{1,k-1-\ell}\right)\nonumber\\
			&=2^{j-1}H_{1,k}+\left(\sum_{\ell=0}^{j-2}(-1)^{\ell+1}\binom{j-1}{\ell+1}2^{j-2-\ell} H_{1,k-1-\ell}\right)\nonumber\\
			&=2^{j-1}H_{1,k}+(-1)^{k}2^{j-3}\left(\sum_{\ell=1}^{j-1}\binom{j-1}{\ell}\psi(k-2,k-1-\ell)\right).\label{eqcaseC001}
			\end{align}
		In addition, 
		\begin{equation}
			\psi(k-2,k-1)=0\quad\text{ and }\quad \binom{j-1}{\ell}=0 \quad \text{for} \quad j\le \ell \le k.\label{eqcaseC002}
		\end{equation}
		Then, by identities \eqref{hiducitemC}, \eqref{eqcaseC001}, \eqref{eqcaseC002}, observation~\eqref{eq12not} and Lemma~\ref{lemaproperpsi} (item $5$), we obtain 			
		\begin{align*}
			H_{1,jk}&=2^{k+j-2}+(-1)^{k}2^{j-3}\left(\sum_{\ell=0}^{k}\binom{j-1}{\ell}\psi(k-2,k-\ell-1)\right)\\
			&=2^{k+j-2}\psi(k+j-2,-1)+(-1)^{k}2^{j-3}\psi(k+j-3,k-1).
		\end{align*}
		Thus, identity~\eqref{firtstepinducb} is satisfied.
		
		\item[$ (D) $] Let $ r\in[1,k-1] $ be fixed and assume as inductive hypothesis that identity~\eqref{firtstepinducb} is satisfied up to $ j-1\geq0 $. Then replacing $ r $ by $ r-1 $ in the inductive hypothesis, we have 
		\begin{align*}
			H_{1,ik+r-1} &= (-1)^{r-1} 2^{k+i-r-1} \psi(k+i-2,r-2)\\
			             &+(-1)^{k+r-1} 2^{i-r-2} \psi(k+i-3,k+r-2),
		\end{align*}
		for all $ 1 \le i \le  j-1 $ and $ r-1\in[0,k-2] $ fixed\footnote{The case $ r-1=0 $ is obtained from Item~$ (C) $ above.}. In particular, this also implies (see~\eqref{eq12not}) that
		\begin{align*}
			H_{1,k+r}&=(-1)^{r} 2^{k-r-1}\psi(k-1,r-1)+(-1)^{k+r} 2^{-r-2} \psi(k-2,k+r-1)\\
			&=(-1)^{r} 2^{k-r-1}\psi(k-1,r-1)
		\end{align*}
		for $ r\in[1,k-1] $ fixed. Thus, by the above two identities and Lemma~\ref{lemaproperty}, item~ ($II$), we get 
		\begin{align}
			H_{b,jk+r}&=2^{j-1}H_{b,k+r}-\sum_{i=1}^{j-1}2^{j-1-i}H_{b,ik+r-1}\nonumber\\
			&=(-1)^{r} 2^{k+j-r-2}\psi(k-1,r-1)\nonumber\\
			&\qquad\qquad+(-1)^{r}2^{k+j-r-2} \sum_{i=1}^{j-1}\psi(k+i-2,r-2)\nonumber\\
			&\qquad\qquad\qquad+(-1)^{k+r} 2^{j-r-3} \sum_{i=1}^{j-1}\psi(k+i-3,k+r-2)\label{identitycprima}		
		\end{align}		
		for all $r, j\in[1,k-1]$, with $ r $ fixed. We note that using the identity~$ (1) $ of Lemma~\ref{lemaproperpsi}, it follows that
		\begin{align}
			&\psi(k-1,r-1)+\sum_{i=1}^{j-1}\psi(k+i-2,r-2)\nonumber\\
			&=\psi(k,r-1)+\sum_{i=2}^{j-1}\psi(k+i-2,r-2)\nonumber\\
			&=\psi(k+1,r-1)+\sum_{i=3}^{j-1}\psi(k+i-2,r-2)\nonumber\\
			&~~\vdots\nonumber\\
			&=\psi(k+j-3,r-1)+\sum_{i=j-1}^{j-1}\psi(k+i-2,r-2)\nonumber\\
			&=\psi(k+j-2,r-1).\label{identitycprima0}
		\end{align}		
		Furthermore, by observation~\eqref{eq12not},
		\begin{align}
			  \psi(k+i-3,k+r-2)&=0&  &\text{for}  &i&<r+1.\label{identitycprima1}
		\end{align}
		Then, by the identities \eqref{identitycprima}, \eqref{identitycprima0}, \eqref{identitycprima1} and Lemma~\ref{lemaproperpsi}, item ($2$), we obtain
		\begin{align*}
			H_{b,jk+r}&=(-1)^{r} 2^{k+j-r-2}\left(\psi(k-1,r-1)+\sum_{i=1}^{j-1}\psi(k+i-2,r-2)\right)\nonumber\\
			&+(-1)^{k+r} 2^{j-r-3} \sum_{i=1}^{j-1}\psi(k+i-3,k+r-2)\\
			&=(-1)^{r}2^{k+j-r-2} \psi(k+j-2,r-1)\\
			&+(-1)^{k+r} 2^{j-r-3} \psi(k+j-3,k+r-1).
		\end{align*}
		Therefore, identity~(\ref{firtstepinducb}) is satisfied.
		
		
		\item[$ (E) $] We fix $ j\in[1,k-1] $ and take as inductive hypothesis the identity~\eqref{firtstepinducb} for $ r-1\geq0 $. In items $ (A) $, $ (C) $ and $ (D) $ above, we proved that identity~\eqref{firtstepinducb} is satisfied on $ H_{ik+r-1} $, for fixed $ r-1\geq 0 $ and all $ i\in[0,k-1] $. Then 
		\begin{align*}
			H_{1,ik+r-1}&=(-1)^{r-1} 2^{k+i-r-1} \psi(k+i-2,r-2)\nonumber\\
			&+(-1)^{k+r-1} 2^{i-r-2} \psi(k+i-3,k+r-2)
		\end{align*}
		holds for all $ 0\le i \le j-1 $. Thus, by the above identity, identity~\eqref{eqitemB1r} and Lemma~\ref{lemaproperty}, item ($II$), it follows that
		\begin{align}
			H_{1,jk+r}&=2^{j}H_{1,r}-\sum_{i=0}^{j-1}2^{j-1-i}H_{1,ik+r-1}\nonumber\\
			&=(-1)^{r} 2^{k+j-r-2}\psi(k-2,r-1)
			+(-1)^{r} 2^{k+j-r-2}\sum_{i=0}^{j-1} \psi(k+i-2,r-2)\nonumber\\
			&+(-1)^{k+r} 2^{j-r-3}\sum_{i=0}^{j-1} \psi(k+i-3,k+r-2)\label{eqitemE1}
		\end{align}
		holds for all $r, j\in[1,k-1]$, with $ j $ fixed. Now, by identity~$ (1) $ of Lemma~\ref{lemaproperpsi} and identity~\eqref{identitycprima0}, we deduce that
		\begin{align}
			&\psi(k-2,r-1)+\sum_{i=0}^{j-1} \psi(k+i-2,r-2)\nonumber\\
			&=\psi(k-1,r-1)+\sum_{i=1}^{j-1}\psi(k+i-2,r-2)=\psi(k+j-2,r-1).\label{eqitemE2}
		\end{align}
		Therefore, identities \eqref{identitycprima1}, \eqref{eqitemE1}, \eqref{eqitemE2} and Lemma~\ref{lemaproperpsi}, item ($2$), lead us to 
		\begin{align*}
				H_{1,jk+r}&=(-1)^{r} 2^{k+j-r-2}\psi(k+j-2,r-1)\nonumber\\
				&+(-1)^{k+r} 2^{j-r-3}\sum_{i=r+1}^{j-1} \psi(k+i-3,k+r-2)\\
				&=(-1)^{r} 2^{k+j-r-2}\psi(k+j-2,r-1)\nonumber\\
				&+(-1)^{k+r} 2^{j-r-3}\sum_{i=0}^{j-r-2} \psi(k+r-2+i,k+r-2)\\
				&=(-1)^{r} 2^{k+j-r-2}\psi(k+j-2,r-1)\nonumber\\
				&+(-1)^{k+r} 2^{j-r-3}\psi(k+j-3,k+r-1),
		\end{align*}
		thus satisfying identity~(\ref{firtstepinducb}).
	\end{enumerate}
	Finally, items $ (A) $, $ (B) $, $ (C) $, $ (D) $ and $ (E) $ above prove the lemma. \qedhere
	
\end{proof}

The following result generalizes the previous one.
\begin{lema}\label{lemaitemiv}
	For all $ k>2 $; $ j,r\in[0,k-1] $ and $ b\in\mathbb{Z}^{+} $, it holds that
	\begin{equation}
		H_{b,jk+r}=\sum_{i=0}^{b} (-1)^{i k+r} 2^{(b-i)k+j-r-b-i-1} \psi(bk+j-b-i-1,ik+r-1).\label{lemaitemiveqqe00}
	\end{equation}
\end{lema}
\begin{proof}
	Here, we proceed by induction in $ b $. The case $ b=1 $ follows from Lemma~\ref{lemaitemiv00}. We state as inductive hypothesis that
	\begin{equation}
			H_{b-1,jk+r}=\sum_{i=0}^{b-1} (-1)^{i k+r} 2^{(b-i-1)k+j-r-b-i} \psi((b-1)k+j-b-i,ik+r-1),\label{eqlemafinal1}
	\end{equation}
	for $ b-1\geq 1 $, $ k>2 $ and $ j,r\in[0,k-1] $. Next, we consider the following cases.
	\begin{enumerate}[leftmargin=*]
		\item[$ (A') $] Let $ j\leq r $. Then, by Lemma~\ref{lemaproperty} (item $III$)  with $ t=j $ and the identities~\eqref{eqlemafinal1},~\eqref{roweqrow1}, we have that
		\begin{align}
			&H_{b,jk+r}=\sum_{\ell=0}^{j}(-1)^\ell \binom{j}{\ell}2^{j-\ell} H_{b-1,(k-1)k+r-\ell}\nonumber\\
			&=\sum_{i=0}^{b-1} (-1)^{i k+r} 2^{(b-i)k+j-r-b-i-1} \sum_{\ell=0}^{j}\binom{j}{\ell} \psi(bk-b-i-1,ik+r-\ell-1).\label{eqlemafinal2}
		\end{align}
		We note that $ j\leq r<ik+r $ for all $ i\geq0 $, and
		\begin{equation}
				\binom{j}{\ell}=0\quad\text{ for } j+1 \le \ell \le  ik+r. \label{eqlemafinal3}
		\end{equation}
		 We use the identities \eqref{eqlemafinal2}, \eqref{eqlemafinal3} and Lemma~\ref{lemaproperpsi}, item ($5$), to obtain 
		 \begin{small}
		 	\begin{align}
		 		H_{b,jk+r}&=\sum_{i=0}^{b-1} (-1)^{i k+r} 2^{(b-i)k+j-r-b-i-1} \sum_{\ell=0}^{ik+r}\binom{j}{\ell} \psi(bk-b-i-1,ik+r-\ell-1)\nonumber\\
		 		&=\sum_{i=0}^{b-1} (-1)^{i k+r} 2^{(b-i)k+j-r-b-i-1} \psi(bk+j-b-i-1,ik+r-1).\label{eqlemafinal4}
		 	\end{align}
		 \end{small}
	 		In fact, $ b(k-2)+j-1<bk+r-1 $ and therefore (see~\eqref{eq12not})
	 		\[
	 			\psi(b(k-2)+j-1,bk+r-1)=0.
	 		\]
	 		So, the sum in identity~\eqref{eqlemafinal4} can be extended to $ b $, fulfilling identity~\eqref{lemaitemiveqqe00} in this case.
	 		
	 		\item[$ (B') $] Let $ r<j  $. Here, we proceed by induction on $ r $.
	 		
	 		\begin{itemize}[leftmargin=*]
	 		\item For $ r=0 $, we apply induction on $ j>0 $. If $ j=1 $, by the identities \eqref{roweqrow1}, \eqref{consecnotationH0} and \eqref{eqlemafinal1} we obtain that
	 		\begin{align*}
	 			H_{b,k} &= 2H_{b,0}-H_{b,-1}=2H_{b-1,(k-1)k}-H_{b-1,(k-2)k+(k-1)}\\
	 			&=\sum_{i=0}^{b-1} (-1)^{i k} 2^{(b-i)k-b-i} \psi(bk-b-i-1,ik-1)\\
	 			&-\sum_{i=0}^{b-1} (-1)^{i k+k-1} 2^{(b-i-1)k-b-i-1} \psi(bk-b-i-2,ik+k-2).
	 		\end{align*}
 			In addition, 
\begin{align*}
 \psi(bk-b-b-1,bk-1)&=0,\\ 
 \psi(bk-b-1,-1)=\psi(bk-b,-1)&=1
\end{align*}
(see \eqref{eq12not}) and by Lemma~\ref{lemaproperpsi}, item ($1$) we obtain
 			\begin{small}
 				\begin{align}
 					&H_{b,k}=\sum_{i=0}^{b-1} (-1)^{i k} 2^{(b-i)k-b-i} \psi(bk-b-i-1,ik-1)\nonumber\\
 					&+\sum_{i=1}^{b} (-1)^{ik} 2^{(b-i)k-b-i} \psi(bk-b-i-1,ik-2)\nonumber\\
 					&=2^{bk-b} \psi(bk-b-1,-1)\nonumber\\
 					&+\sum_{i=1}^{b} (-1)^{ik} 2^{(b-i)k-b-i} \left(\psi(bk-b-i-1,ik-1)+\psi(bk-b-i-1,ik-2)\right)\nonumber\\
 					&=\sum_{i=0}^{b} (-1)^{ik} 2^{(b-i)k-b-i}\psi(bk-b-i,ik-1)\label{eqitemBprime001}
 				\end{align}
 			\end{small}
 			satisfying identity~\eqref{lemaitemiveqqe00}.
 			
 			\item Now, we assume that identity~\eqref{lemaitemiveqqe00} is satisfied for $ j-1 $, when $ r=0 $. We obtain that
 			\begin{align}
 				&H_{b,jk}=2H_{b,(j-1)k}-H_{b,(j-1)k-1}=2H_{b,(j-1)k}-H_{b,(j-2)k+k-1}\nonumber\\
 				&=-H_{b,(j-2)k+k-1}+\sum_{i=0}^{b} (-1)^{i k} 2^{(b-i)k+j-b-i-1} \psi(bk+j-b-i-2,ik-1).\label{eqitemBprime01}
 			\end{align}
 			We also note that $ j-2<k-1 $, so item $ (A') $ above implies that
 			\begin{equation}
 				H_{b,(j-2)k+k-1}=\sum_{i=0}^{b} (-1)^{ik+k-1} 2^{(b-i-1)k+j-b-i-2} \psi(bk+j-b-i-3,ik+k-2).\label{eqitemBprime02}
 			\end{equation}
 			Then, since 
 	\begin{align*}
 		\psi(bk+j-2b-3,bk+k-2)& = 0\\
 		\psi(bk+j-b-2,-1)=\psi(bk+j-b-1,-1)&=1
	\end{align*} 			
 			(see~\eqref{eq12not}), we use identities \eqref{eqitemBprime01}, \eqref{eqitemBprime02} and Lemma~\ref{lemaproperpsi}, item ($1$) to obtain
 			\begin{align*}
 				H_{b,jk}&=2^{bk+j-b-1} \psi(bk+j-b-2,-1)\\
 				&+\sum_{i=1}^{b} (-1)^{ik} 2^{(b-i)k+j-b-i-1} \psi(bk+j-b-i-2,ik-2)\\
 				&+\sum_{i=1}^{b} (-1)^{i k} 2^{(b-i)k+j-b-i-1} \psi(bk+j-b-i-2,ik-1)\\
 				&=\sum_{i=0}^{b} (-1)^{ik} 2^{(b-i)k+j-b-i-1} \psi(bk+j-b-i-1,ik-1).
 			\end{align*}
Therefore identity~\eqref{lemaitemiveqqe00} holds in this case. This concludes the induction on $ j $ when $ r=0 $. Thus, the base case of induction on $ r $ is done.  			
 			 			
 			\item Let $ j=r+1 $. Item $ (A') $ above implies that  
 			\begin{equation}
 				H_{b,k+\ell}=\sum_{i=0}^{b} (-1)^{i k+\ell} 2^{(b-i)k-\ell-b-i} \psi(bk-b-i,ik+\ell-1)\label{eqitemBprime03}
 			\end{equation}
 			for $ 1 \le \ell \le r $. By Lemma~\ref{lemaproperty}, item ($III$) with $ t=r $, and the identities \eqref{eqitemBprime001}, \eqref{eqitemBprime03}, we obtain 
 			\begin{align}
  H_{b,(r+1)k+r}&=\sum_{\ell=0}^{r} (-1)^{r-\ell} \binom{r}{r-\ell}2^{\ell} H_{b,k+\ell}\nonumber\\
 				&=\sum_{\ell=0}^{r}  \binom{r}{\ell}\left(\sum_{i=0}^{b} (-1)^{i k+r} 2^{(b-i)k-b-i} \psi(bk-b-i,ik+\ell-1)\right)\nonumber\\
 				&=\sum_{i=0}^{b} (-1)^{i k+r} 2^{(b-i)k-b-i} \sum_{\ell=0}^{r} \binom{r}{r-\ell}\psi(bk-b-i,ik+r-\ell-1).\label{eqitemBprime04}
 			\end{align}
 			Furthermore, $ r\leq ik+r $ for $ i\geq 0 $ and 
 			\begin{equation}
 		\binom{r}{r-\ell}=\binom{r}{\ell}=0,\quad\text{ for }\quad r+1 \le \ell \le ik+r\quad\text{ when }\quad i\geq1.\label{eqitemBprime05}
 			\end{equation}
 			Then, by identities \eqref{eqitemBprime04}, \eqref{eqitemBprime05} and Lemma~\ref{lemaproperpsi}, item ($5$), we arrive at
 			\begin{small}
 				\begin{align*}
 					H_{b,(r+1)k+r}&=\sum_{i=0}^{b} (-1)^{i k+r} 2^{(b-i)k-b-i} \sum_{\ell=0}^{ik+r} \binom{r}{\ell}\psi(bk-b-i,ik+r-\ell-1)\\
 					&=\sum_{i=0}^{b} (-1)^{i k+r} 2^{(b-i)k-b-i} \psi(bk-b-i+r,ik+r-1).
 				\end{align*} 				
 			\end{small}
 			The above identity verifies identity~(\ref{lemaitemiveqqe00}) in this case, and we finish the base step of induction on $ j $ and on $ r $.
 			
 		\item We assume by the inductive hypothesis that identity~(\ref{lemaitemiveqqe00}) is fulfilled for $ r-1 $ and up to $ j-1 $; i.e.,
 			\begin{equation}
 				H_{b,\ell k+r-1}=\sum_{i=0}^{b} (-1)^{i k+r-1} 2^{(b-i)k+\ell-r-b-i} \psi(bk+\ell-b-i-1,ik+r-2)\label{eqitemBprime06}
 			\end{equation}
 			for $ r \le \ell \le  j-1 $. Also, by item $ (A') $ above, we obtain that identity~\eqref{eqitemBprime06} holds for $0 \le \ell \le r-1 $. In fact, it implies that 
 			\begin{equation*}
 				H_{b,r}=\sum_{i=0}^{b} (-1)^{i k+r} 2^{(b-i)k-r-b-i-1} \psi(bk-b-i-1,ik+r-1).\label{eqitemBprime07}
 			\end{equation*} 			
 			Then, by Lemma~\ref{lemaproperty}, item ($II$) and identity~\eqref{eqitemBprime06}, we arrive at
 			\begin{align}
 				H_{b,jk+r}&=2^{j}H_{b,r}-\sum_{\ell=0}^{j-1}2^{j-1-\ell}H_{b,\ell k+r-1}\nonumber\\
 				&=\sum_{i=0}^{b} (-1)^{i k+r} 2^{(b-i)k+j-r-b-i-1} \psi(bk-b-i-1,ik+r-1)\nonumber\\
 				&+\sum_{i=0}^{b} (-1)^{i k+r} 2^{(b-i)k+j-r-b-i-1} \sum_{\ell=0}^{j-1}\psi(bk+\ell-b-i-1,ik+r-2)\label{eqitemBprime08}
 			\end{align}
 		for all $r, j\in[1,k-1]$\footnote{The case $ r=0 $ is made in the first item of this item~$ (B') $.}. We observe, by Lemma~\eqref{lemaproperpsi}, item ($1$), that
 			\begin{align}
 			&\psi(bk-b-i-1,ik+r-1)+\sum_{\ell=0}^{j-1}\psi(bk+\ell-b-i-1,ik+r-2)\nonumber\\
 			&=\psi(bk-b-i,ik+r-1)+\sum_{\ell=1}^{j-1}\psi(bk+\ell-b-i-1,ik+r-2)\nonumber
 		\end{align}
    	\begin{align}
 			&=\psi(bk-b-i+1,ik+r-1)+\sum_{\ell=2}^{j-1}\psi(bk+\ell-b-i-1,ik+r-2)\nonumber\\
 			&~\vdots\nonumber\\
 			&=\psi(bk-b-i+j-2,ik+r-1)+\psi(bk+j-b-i-2,ik+r-2)\nonumber\\
 			&=\psi(bk-b-i+j-1,ik+r-1).\label{eqitemBprime09}
 		\end{align}
 		Then, by the identities \eqref{eqitemBprime08} and \eqref{eqitemBprime09}, we obtain
 		\begin{align*}
 			H_{b,jk+r}&=\sum_{i=0}^{b} (-1)^{i k+r} 2^{(b-i)k+j-r-b-i-1}\Bigg(\psi(bk-b-i-1,ik+r-1)\\
 			&+\left.\sum_{\ell=0}^{j-1}\psi(bk+\ell-b-i-1,ik+r-2)\right)\\
 			&=\sum_{i=0}^{b} (-1)^{i k+r} 2^{(b-i)k+j-r-b-i-1} \psi(bk-b-i+j-1,ik+r-1)
 		\end{align*}
 		and identity~\eqref{lemaitemiveqqe00} is satisfied.  This finishes the induction on both $ j $ and $ r $ thus completing the proof of this lemma.
 	\end{itemize}
	\qedhere
	\end{enumerate}
 
\end{proof}

Now we return to the proof of item $ (iv) $. Replacing $ b(k-1)+j $ by $ m $ and $ r-1 $ by $ r $ in Lemma~\ref{lemaitemiv}, we obtain that 
\begin{align}
	H_{mk+r}&=H_{(b(k-1)+j)k+r}=H_{b,jk+r+1}\label{eqcaseivs220}\\
			&=\sum_{i=0}^{b} (-1)^{i k+r+1} \psi(m-i-1,ik+r) 2^{m-i(k+1)-r-2},\label{eqcaseivs22}
\end{align}
for all $ k>2 $, $ j\in[0,k-1] $, $ r\in[-1,k-2] $ and $ b\in\mathbb{Z}^{+} $. By identity~\eqref{roweqrow1}, we have
\begin{equation*}
	H_{b,r+1}=H_{b-1,(k-1)k+r+1}\quad\text{for all } ~-1\leq r\leq k-2 \quad\text{and }~ b\in\mathbb{Z}^{+},
\end{equation*}
so we can assume that $ j\in\left[0,k-1\right) $ (in order not to repeat elements of the sequence in \eqref{eqcaseivs220}). Then
\[
	\frac{m}{k-1}-1<\frac{m-j}{k-1}\leq\frac{m}{k-1} \quad\text{for }\quad k-1> 1,
\]
therefore $ \left\lfloor \frac{m}{k-1}\right\rfloor -1 < b \leq \left\lfloor \frac{m}{k-1}\right\rfloor $ and $ b= \left\lfloor \frac{m}{k-1}\right\rfloor $. In conclusion, item $ (iv) $ follows from replacing $ b $ by $ \left\lfloor \frac{m}{k-1}\right\rfloor $ in identity~\eqref{eqcaseivs22}\footnote{Note that $ m=b(k-1)+j\geq k-1 $.}.\\

This completes the proof of Theorem \ref{MT}.

\section{Proof of Conjecture \ref{conjectute}}

We distinguish two cases according to the parity of $k$. Furthermore, since item $(i)$ of Theorem \ref{MT} determines the set of zeros of $H^{(k)}$ in \eqref{zeromultcont} and items $(ii)$--$(iii)$ don't provide zeros for $H^{(k)}$, we can assume that $m \ge k-1$.     

\subsection{Case  $k$ even}
In this case, by item $(iv) $ in Theorem \ref{MT}, we get that for all $ r\in[-1,k-2] $ and $ m\geq k-1 $, 
\begin{align*}
	H_{mk+r}&=(-1)^{r+1}\sum_{i=0}^{l} \left[\binom{m-i-1}{ik+r}+\binom{m-i}{ik+r+1}\right]2^{m-i(k+1)-r-2},
\end{align*}
where $ l=m-1 $ if $ k=2 $, and $ l = \lfloor m/(k-1) \rfloor $ if $ k>2 $. So, for $ n = mk + r\geq k^2-k-1 $, we have $ H_{n}< 0 $ if $ n $ is even, while $H_{n} > 0 $ if $ n $ is odd, given that $n$ and $r$ have same parity because $k$ is even. 

Thus, item $(i) $ of Theorem \ref{MT} implies that in this case the zero--multiplicity  of $ H^{(k)}$ is exactly $ k (k-1) / 2 $, confirming Conjecture \ref{conjectute} in this case.

\subsection{Case  $k$ odd} 

Since we verified in our previous work \cite{zero-multiplicity} that $\#{\mathcal Z}({H^{(k)}})= k(k-1)/2$ for $k \in [4,500]$, we will assume that $ H_{n} = 0$, for $n=mk+r$, with $m \ge k-1, r\in [-1,k-2]$ and $k > 500$ odd. 

\subsubsection{A lower bound for $n$ in terms of $k$} 

By item $(iv) $ of Theorem~\ref{MT}, we have after simplifying a factor of $2^{m-r-2}$, that  
\begin{equation}
	\sum_{i=0}^{l} (-1)^{i}\psi(m-i-1,ik+r)2^{-i(k+1)}=0, \qquad {for} ~~ l=\left\lfloor m/(k-1)\right\rfloor.\label{eqsumzero}
\end{equation}
We note that in the sum of equation~\eqref{eqsumzero}, the term for $ i=0 $ is non--zero since $ r\leq k-2<k-1\leq m $, which implies that $ r\leq m-1 $ and therefore $ \psi(m-1,r)\neq 0 $. If the case $ i=0 $ were the only for which $\psi(m-i-1,ik+r)$ is non--zero, we would have that equation~\eqref{eqsumzero} has no solution and our problem of zero multiplicity for the odd $ k $ case would be solved. So, we must assume that at least two terms of the above sum (the first with $ i=0 $ and the second with $ i $ odd) are non--zero.

Next, we take $ l'$ to be the largest index $ i>0 $ for which the $ i $-th term of the sum in equation~\eqref{eqsumzero} is non--zero and see that equation~\eqref{eqsumzero} takes the form 
\[
	\sum_{i=0}^{l'} (-1)^{i}\psi(m-i-1,ik+r)2^{-i(k+1)}=0, \qquad {\rm for} ~~ 0< l'\leq \left\lfloor m/(k-1)\right\rfloor,
\]
where $ \psi(m-l'-1,l'k+r) $ is non--zero.

Separating the case $i=l'$, we get
\[
\psi(m-l'-1,l'k+r) = 2^{k+1}\sum_{i=0}^{l'-1} (-1)^{l'+i}\psi(m-i-1,ik+r)2^{(l'-i-1)(k+1)},
\]
with $ l'-i-1\geq0 $ for all $ i=0,\ldots,l'-1 $.
Hence,
$$
2^{k+1}\mid \psi(m - l' - 1, l'k+r),\qquad {\text{\rm where}}\qquad l'\leq \lfloor m/(k-1)\rfloor.
$$ 
Now, 
$$
\psi(v,w)=\binom{v}{w}+\binom{v+1}{w+1}=\binom{v}{w}\left(1+\frac{v+1}{w+1}\right).
$$
Kummer \cite{Kummer} proved that $\nu_2\left(\binom{v}{w}\right)$ equals the number of carries when adding $w$ with $v-w$ in base $2$. Here, for a nonzero integer $m$, $\nu_2(m)$ is the exponent of $2$ in the factorization of $m$. 
In particular,
$$
\nu_2\left(\binom{v}{w}\right)\le \frac{\log v}{\log 2}+1.
$$
Hence, 
\begin{eqnarray*}
	\nu_2 (\psi(v,w))  \le \nu_2\left(\binom{v}{w}\right)+\nu_2(v+w+2) &\le& \frac{\log v}{\log 2}+1+\frac{\log(v+w+2)}{\log 2}\\
	& \le & \frac{2\log v}{\log 2}+2,
\end{eqnarray*}
where for the last inequality we used the fact that we may assume that $w\le v-2$ (indeed, if $w=v$, then $\psi(v,w)=2$, so the above bound holds while if $w=v-1$, then $\psi(v,w)=2v+1$ is odd so the above bound again holds).
Since 
$$
2^{k+1}\mid \psi(v,w), \quad {\rm for} ~~ (v,w) = (m-l'-1, l'k+r), 
$$
we get $\nu_2(\psi(v,w))\ge k+1$. Hence, 
$$
k+1\le \frac{2\log v}{\log 2}+2,\quad {\text{\rm therefore}}\quad v\ge 2^{(k-1)/2}.
$$
Since $m-l'-1\le n$, we get that
\begin{eqnarray}\label{lowerbound n-k}
2^{(k-1)/2}\le n.
\end{eqnarray}

\subsubsection{An upper bound for $n$ in terms of $k$} 

We now review the preliminary work \cite{zero-multiplicity} on zero--multiplicity of $H^{(k)}$, to obtain a better upper bound for $n$ on $k$, which we then combine with the above lower bound \eqref{lowerbound n-k}. 
In \cite{zero-multiplicity}, we used a Binet--type formula of $H_n$, namely 
$$
H_n = f_{k}(\alpha_{1})\alpha_{1}^{-(n+1)} + \cdots + f_{k}(\alpha_{k-1})\alpha_{k-1}^{-(n+1)} + f_{k}(\alpha_{k})\alpha_{k}^{-(n+1)}
$$
where $\alpha_{1},\ldots,\alpha_{k-1},\alpha_{k}$ are all the roots of the characteristic polynomial of $F^{(k)}$, with $\A_1>1$ being the only real root and the remaining $k-1$ roots lie inside the unit disk. Furthermore, 
$$
\alpha_1 > |\alpha_{2}|\geq\cdots\geq|\alpha_{k-1}|\geq|\alpha_{k}| \quad {\rm and} \quad f_k(z) := (z-1)/\left(2+(k+1)(z-2)\right).
$$ 
Thus, if $H_{n}=0$, then 
\begin{eqnarray}\label{upperbound}
|\Lambda| := \left|1+\left(\frac{f_{k}(\alpha_{k-1})}{f_{k}(\alpha_{k})}\right)\left(\frac{\alpha_{k}}{\alpha_{k-1}}\right)^{n+1}\right| 
         & = &\sum_{i=1}^{k-2}\left|\frac{f_{k}(\alpha_{i})}{f_{k}(\alpha_{k})}\right|\left|\frac{1}{\alpha_{i}/\alpha_{k}}\right|^{n+1} \nonumber\\ 
	     & < & \frac{3(k-2)/(k-1)}{\left|\alpha_{k-2}/\alpha_{k}\right|^{n+1}}\left|\frac{1}{f_{k}(\alpha_{k})}\right|\label{upperbound-1}\\
	     & < &\frac{13k(3k+1)}{\left|\alpha_{k-2}/\alpha_{k}\right|^{n+1}}.\label{upperbound-2}\nonumber
\end{eqnarray}
Using an argument involving lower bounds for nonzero linear forms in logarithms of complex algebraic numbers, we found 
\begin{eqnarray}\label{lowerboundMat}
|\Lambda| >  \exp(-5 \cdot 10^{13} \times k^7\log (n+1)(\log{k})^2).
\end{eqnarray}
The following result will be fundamental to combine inequalities \eqref{upperbound} and \eqref{lowerboundMat}, in order to find a better upper bound for $n$ on $k$ that one given in \cite{zero-multiplicity}.
The next lemma is Theorem 1 in \cite{GGL-preprint} and represents an improvement over the analogous inequality in \cite{Dubickas}. 
\begin{lema}\label{lm |ai|/|aj|}
The inequality
\begin{equation*}
\frac{|\alpha_i|}{|\alpha_j|}>1+\frac{1}{10k^{9.6} (\pi/e)^k}\quad {\text{\rm holds~for~all}}\quad 1\le i<j\le (k-1)/2
\end{equation*}
and all $k\ge 4$.
\end{lema}

Hence, by inequalities \eqref{upperbound}, \eqref{lowerboundMat} and Lemma \ref{lm |ai|/|aj|} 
\begin{eqnarray}\label{upperbound n-k}
n < 3\cdot 10^{14} k^{17.6} (\log k)^3 (\pi/e)^k.
\end{eqnarray}

\subsubsection{Absolute bounds for $n, k$ and final conclusion} 

Combining \eqref{lowerbound n-k} and \eqref{upperbound n-k}, we get
$$
\left(\frac{{\sqrt{2}}}{(\pi/e)}\right)^{k}<3\sqrt{2}\cdot 10^{14} k^{17.6} (\log k)^3,
$$
showing that $k \le 790$. 

By \eqref{upperbound-1} and \eqref{lowerboundMat}, we have
\begin{equation}\label{des n-k}
5 \cdot 10^{13} \times k^7\log (n+1)(\log{k})^2 > \log\left|\cfrac{f_{k}(\alpha_{k})}{3}\right|+(n+1)\log\left|\cfrac{\alpha_{k-2}}{\alpha_{k}}\right|.
\end{equation}
Thus, using the fact that $k \in [501,789]$, we obtain computationally an upper bound on $n$ in each case:
\begin{equation*}\label{firstinequality}
{\rm If} ~~ H_{n} = 0 \quad {\rm and} \quad k\in[501, 789], \quad {\rm then}\quad n \in 2.5\cdot10^{45}.
\end{equation*}

Returning to inequality \eqref{lowerbound n-k}, we have $2^{(k-1)/2} < n < 2.5\cdot10^{45}$ which leads to $k \le 517$. Finally, we return once again to \eqref{des n-k} where now we get $n < 3.5\cdot10^{43}$ for all odd $k \in[501,517]$. Then, by \eqref{lowerbound n-k}, we conclude that $k<500$, contradicting our initial assumption about $k$.

This completes the proof of Conjecture \ref{conjectute} for the case of $k$ odd.

\textbf{Acknowledgement.}
J. G. thanks the Universidad del Valle for support during his master's studies.
C.A.G. was supported in part by Project 71327 (Universidad del Valle).

\end{document}